\documentclass{amsart}

\numberwithin{equation}{section}

\usepackage{amssymb}
\usepackage{enumerate, xspace}
\usepackage{dsfont}
\usepackage{changebar}
\usepackage[colorlinks]{hyperref}

\newtheorem{thmm}{Theorem}
\newtheorem{thm}{Theorem}[section]
\newtheorem{lem}[thm]{Lemma}
\newtheorem{cor}[thm]{Corollary}

\newtheorem{conv}[thm]{Convention}

\newtheorem{rem}[thm]{Remark}

\newcommand\cB{{\mathcal B}}
\newcommand\cC{{\mathcal C}}

\newcommand\cF{{\mathcal F}}

\newcommand\cL{{\mathcal L}}

\newcommand\cN{{\mathcal N}}
\newcommand\cO{{\mathcal O}}

\newcommand\cQ{{\mathcal Q}}
\newcommand\cR{{\mathcal R}}

\newcommand\bE{{\mathbb E}}

\newcommand\bL{{\mathbb L}}

\newcommand\bN{{\mathbb N}}

\newcommand\bP{{\mathbb P}}

\newcommand\bR{{\mathbb R}}

\newcommand\bT{{\mathbb T}}

\newcommand\bZ{{\mathbb Z}}

\newcommand\ve{\varepsilon}

\newcommand\vf{\varphi}
\newcommand\Var{\Sigma^2}

\newcommand\btau{{\boldsymbol{\tau}}}

\newcommand\Id{{\mathds{1}}}

\newcommand{\bfP}{{\bf P}}
\newcommand{\Leb}{{\mathfrak m}}
\newcommand{\mfa}{{\mathfrak A}}

\begin{document}

\title[Quenched CLT]{Quenched CLT for random toral automorphism}

\author{Arvind Ayyer}
\address{Arvind Ayyer \\
Department of Physics \\
Rutgers University\\
136 Frelinghuysen Road\\
Piscataway, NJ 08854, USA.}
\email{ayyer@physics.rutgers.edu}

\author{Carlangelo Liverani}
\address{Carlangelo Liverani\\
Dipartimento di Matematica\\
II Universit\`{a} di Roma (Tor Vergata)\\
Via della Ricerca Scientifica, 00133 Roma, Italy.}
\email{{\tt liverani@mat.uniroma2.it}}

\author{Mikko Stenlund}
\address{Mikko Stenlund\\
Department of Mathematics \\
Rutgers University \\
110 Frelinghuysen Road \\
Piscataway, NJ 08854, USA.}
\email{mstenlun@math.rutgers.edu}

\date{\today}
\begin{abstract}
We establish a quenched Central Limit Theorem (CLT) for a smooth
observable of random sequences of iterated linear hyperbolic maps on the torus.
To this end we also obtain an annealed CLT for the same
system. We show that, almost surely, the variance of the quenched
system is the same as for the annealed system. Our technique is the
study of the transfer operator on an anisotropic Banach space
specifically tailored to use the cone condition satisfied by the maps.

\end{abstract}
\maketitle

%%%PAPER
\section{introduction}
The issue of limit laws in dynamical systems has been widely explored in
the last decades and it has a clear relevance for physical applications. A
prime example of a physically relevant system is the study of the
statistical behavior of a Lorenz gas with randomly distributed obstacles.
The case of periodic obstacles is known to be ergodic. 
This follows from the recurrence \cite{Sim}, which in turns
follows from the CLT, \cite{Co, Sm}, which has been proved in \cite{bsc} (see also \cite{DSV} for more refined results on these issues).
See \cite{Le} for more details and for the treatment of some (locally)
aperiodic cases.
On the contrary the random case (albeit one may na\"{\i}vely think of it as an easier
case) stands as a challenge.

If one considers the simplest possibility (the random position of the obstacles
is a small i.i.d.\@ perturbation of a periodic configuration) then, by
Poincar\'e section, one is readily reduced to considering a random sequence of
hyperbolic symplectic maps. Yet, such a sequence of maps is not i.i.d.\@
due to the presence of recollisions. Recollisions are notoriously a source of serious problems in the study of gases but, quite surprisingly, even disregarding the recollision problem (i.e. for the i.i.d.\@ case), the problem is poorly understood.

In this paper we address the easiest setting in which such a situation
occurs:  an i.i.d.\@ sequence of smooth uniformly hyperbolic symplectic maps. To make the
presentation as clear as possible we will steer away from the full
generality in which the present results can be obtained (although we will comment on it) and
we will consider an i.i.d.\@ sequence of linear  two dimensional toral
automorphisms. Exponential decay of correlations has been shown in this setting in \cite{AS}.

For such a model we will show that the time-$N$ average of any smooth
zero mean observable has Gaussian fluctuations of order $\sqrt N$ for
almost every sequence of maps. Moreover, we identify the variance of such
Gaussian fluctuations.

Similar, but less complete, results are obtained in \cite{Ba} where the
Gaussian nature of the fluctuations is established for each sequence of maps but
neither the amplitude nor the variance is investigated.

The paper is self-contained  and organized as following: Section \ref{sec:results} contains
the precise description of the model we will discuss and states the main
results of the paper. Such results depend on the understanding of the
ergodic properties of sequences of maps. These are investigated in Section
\ref{sec:gap} where the needed ergodic properties are related to the
spectral properties of transfer operators viewed on appropriate Banach
spaces in the spirit of the line of research started with \cite{BKL}.
Next, in Section \ref{sec:average}, we use the above results to establish
a CLT averaged over the environment for a class of systems larger than the
ones at hand but necessary to handle the quenched case. The latter is
dealt with in Section \ref{sec:quenched} using an approach inspired by 
works on random walks in random environments; see \cite{DKL} and references
therein.

\begin{conv}
\label{conv:contants}
In this paper we will use $C$ to designate a generic constant
depending only on the choice of the maps $\{T_i\}$ below. We will use instead $C_{a,b,c,\dots}$ for constants depending
also on the parameters $a,b,c,\dots$. Consequently, the actual
numerical value of such constants may vary from one occurrence to the
next. On the contrary we will use $C_1, C_2, \dots$, to designate
constants whose value is fixed throughout the paper.
\end{conv}

\subsection*{Acknowledgments}
We would like to thank Joel Lebowitz for posing the problem and
Dmitry Dolgopyat for communicating to one of us (CL) reference \cite{Ba}.
MS would like to thank the Finnish Cultural Foundation for funding. MS
and AA were supported in part by NSF DMR-01-279-26 and AFOSR AF
49620-01-1-0154. CL would like to thank the Courant Institute where
he was visiting when this work started.

\section{The model and the results}\label{sec:results}
Let us consider two\footnote{In fact, the following would hold almost
verbatim also for any larger collection of matrices.}  matrices
$\{A_i\}_{i=0}^{1}\in SL(2,\bN)$  and define the toral automorphisms
$T_ix=A_ix\mod 1$. Let $\wp\in [0,1]$ and set $p_0=\wp$, $p_1=1-\wp$. We
can then introduce the Markov operator $Q_\wp:L^\infty(\bT^2,\bR)\to L^\infty(\bT^2,\bR)$ defined by 
\[
Q_\wp g(x)=\sum_{i=0}^{1}p_ig(T_i(x)).
\]
Such an operator defines a Markov Process. To describe it we consider
the space of trajectories $\Omega_*:=(\bT^2)^{\bN}$ endowed with the
product topology and, letting $(x_0,x_i,\dots)$ be a general element
in $\Omega_*$, we have the obvious dynamics $\btau:\Omega_*\to\Omega_*$ defined
by $\btau(x_0,x_1,\dots)=(x_1,\dots)$. For each initial measure
$\mu$ on $\bT^2$, the above Markov process defines a Borel probability measure
$P_\mu$ on $\Omega_*$. Let $\bE_{P_\mu}$ be the expectation with respect
to such a measure. Then
\[
\begin{split}
&\bE_{P_\mu}(g(x_0))=\int g(\xi)\mu(d\xi)\\
&\bE_{P_\mu}(g(x_{i+1})\;|\; x_i)=Q_\wp g(x_i).
\end{split}
\]
The measure $P_\mu$ is supported on a very small set of
trajectories: $P_\mu$-almost surely
$x_{i+1}\in \{T_0x_i,T_{1}x_i\}$. Thus, if we consider
$\Sigma:=\{0,1\}^{\bN\setminus\{0\}}$, we have
\[
P_\mu(\cup_{\omega\in\Sigma}\{(x_0, T_{\omega_1}x_0,
T_{\omega_2}T_{\omega_1}x_0, \dots)\})=1.
\]
In other words we can define the probability space $\Omega=\Sigma\times
\bT^2$ (again equipped with the product topology), the map $F:\Omega\to\Omega$ defined by
\[
F(\omega, x)=(\tau\omega, T_{\omega_1}x),
\]
where $(\tau\omega)_i=\omega_{i+1}$, and the measure ${\bf
P}_{\wp,\mu}=\bP_{\wp}\times \mu$, where $\bP_{\wp}$ is the Bernoulli
measure with probability $\wp$ of having zero. We will denote by
$\bE_{\wp}$ the expectation with respect to $\bP_\wp$. 
Note that if $\mu$ is simultaneously $T_0$ and $T_1$ invariant,
then ${\bf P}_{\wp,\mu}$ is invariant for the map $F$. Since the maps are
symplectic, this happens for the normalized Lebesgue measure $\Leb$. Let us set ${\bf P}_\wp:={\bf P}_{\wp,\Leb}$ and call $\bE_{{\bf P}_{\wp}}$ the corresponding expectation. 
Finally, we define the map $\Psi:\Omega\to\Omega_*$ by
\[
\Psi(\omega, x):=(x, T_{\omega_1}x, T_{\omega_2}T_{\omega_1}x, \dots).
\]
It is then easy to verify that
$\btau^k(\Psi(\omega,x))=\Psi(F^k(\omega,x))$ for all
$(\omega,x)\in\Omega$, $k\in\bN$,
and $\bE_{P_\mu}(h)=\bE_{{\bf P}_{\wp,\mu}}(h\circ
\Psi)$ for each continuous function $h:\Omega_*\to \bR$,\footnote{Indeed, if $h(x_0,x_1,\dots)=g(x_n)$, then
$\bE_{P_\mu}(h)=\mu(Q_\wp^n g)=\bE_{{\bf P}_{\wp,\mu}}(h\circ \Psi)$. On
the other hand if we have already the equality for functions depending
on $n$ variables, we can write $h(x_0,\dots,x_n,x_{n+1})=g_{x_0,\dots,
x_n}(x_{n+1})$ and, by induction,
\[
\begin{split}
\bE_{P_\mu}(h)&=\bE_{P_\mu}(\bE_{P_\mu}(h\;|\; x_1,\dots, x_n))=\bE_{P_\mu}(Q_\wp g_{x_1,\dots,
x_n}(x_n))\\
&=\bE_{{\bf P}_{\wp,\mu}}\left(\sum_i p_i
h(\Psi(\omega,x)_0, \dots, \Psi(\omega, x)_n, T_i\Psi(\omega,
x)_n)\right)\\
&=\bE_{{\bf P}_{\wp,\mu}}(\bE_{{\bf P}_{\wp,\mu}}(h\circ \Psi\;|\; x,
\omega_1, \dots,\omega_n))=\bE_{{\bf P}_{\wp,\mu}}(h\circ \Psi).
\end{split}
\]
The assertion follows then by the density of the local functions among the
continuous ones.}
that is the two Dynamical Systems $(\Omega_*,\btau, P_\Leb)$ and
$(\Omega, F, {\bf P}_\wp)$  are isomorphic and so are the $\sigma$-algebras
$\cF_k=\sigma\text{-}\{x,\omega_1,\dots,\omega_k\}$ and
$\tilde\cF_k=\sigma\text{-}\{x_0,\dots,x_k\}$. We will use the two processes above
interchangeably as far as the study of measure theoretical properties is
concerned.

For each function 
$f\in\cC^\infty(\bT^2,\bR)$, $f\not\equiv 0$, such that $\Leb(f)=0$ we can then define the random
variables $X_k(\omega,x):=f(\pi\circ F^k(\omega,x))$, where $\pi(\omega,x):= x$. We are interested in studying the $\bP_\wp$-almost sure asymptotic behavior, as $N\to\infty$, of the random variables
\[
S_N(\omega):=\sum_{k=0}^{N-1}X_k(\omega,\cdot).
\]

The first relevant fact lies in the following Lemma.
\begin{lem}\label{lem:ergodic}
The dynamical system  $(\Omega, F, {\bf P_\wp})$ is ergodic.
\end{lem}
\begin{proof}
By the above discussion the ergodicity of $(\Omega, F, {\bf P_\wp})$ is
equivalent to the ergodicity of the stationary Markov process $P_\Leb$.
It is well known that the ergodicity of such a process is equivalent to
the fact that $Q_\wp g=g$ implies $g=\text{constant}$ for each bounded measurable $g$.
In section \ref{sec:gap} we will see (Corollary \ref{cor:trivial})
that there exists $p,q >0$ such that,
for each $f\in\cC^{p+2d+1}, g\in\cC^{q}$, holds 
\begin{equation}\label{eq:mixing}
\lim_{n\to\infty}\Leb(fQ_\wp^ng)=\Leb(f)\Leb(g).
\end{equation}
Taking $f\in\cC^{p+2d+1}$ and $g\in  L^\infty$, we can choose a sequence
$(g_j)_{j=1}^\infty\subset\cC^{q}$ that converges to $g$ in $L^1$ and we obtain
\eqref{eq:mixing} also for such functions. But this means that $Q_\wp g=g$
implies $\Leb(f g)=\Leb(f)\Leb(g)$ for each $f\in\cC^{p+2d+1}$ which readily
implies that $g$ is constant.
\end{proof}
Since $\Leb(S_N(\omega))=0$, thanks to the previous Lemma, we can apply the
Birkhoff Ergodic Theorem and obtain
\[
\lim_{N\to\infty}\frac 1N S_N=0 \quad {\bf P_\wp}-\text{almost surely}.
\]
The next step is to investigate the variable $N^{-\frac 12}S_N$ and prove
that it satisfies a (quenched) CLT.  
\begin{thmm}\label{thm:clt-quenched}
For each $f\in\cC^\infty(\bT^2,\bR)$,\footnote{In the proof we use $f\in \cC^r(\bT^2,\bR)$ for $r$ large enough. Yet, since our bounds for $r$ are far from optimal (nor do we strive to optimize them) we see no point in giving an explicit bound for $r$.} $f\not\equiv 0$, and for each $\wp\in[0,1]$ there exists $\Var_\wp\in\bR_+$ such that,  for
$\bP_\wp$-almost all sequences $\omega$ holds\footnote{By $\cN\bigl(0,
\Var\bigr)$ we mean the centered Gaussian random variable with variance
$\Var$. The symbol $\Rightarrow$ stands for convergence in
distribution. As usual, $\cN(0,0)$ stands for the measure concentrated
at zero.}
\[
N^{-\frac 12}S_N\Rightarrow
\cN\bigl(0, \Var_\wp\bigr)\quad under\  \Leb.
\]
In fact, $\Var_\wp$ depends analytically on $\wp$. Moreover, if $f$ is not a simultaneous continuous coboundary\footnote{By a {\em simultaneous continuos coboundary} for a set of maps $\{T_i\}$ we mean that there exists a continuous function $g$ such that $f=g-g\circ T_i$, for each map $\{T_i\}$.} for each admissible\footnote{The map $T_i$ is admissible if it appears with nonzero probability with respect to $\bP_\wp$. In our case the admissible maps are $\{T_0,T_1\}$ unless $\wp\in\{0,1\}$.} map $T_i$, then $\Sigma_\wp^2>0$.
\end{thmm}
\begin{rem} \label{livsic} Imposing $\Sigma_\wp^2>0$ clearly excludes fewer observables in the random case than the deterministic one.
From Theorem~\ref{thm:clt-quenched} and the classical Livschitz Theorem \cite{KH} easily follows
that $\Sigma^2_\wp=0$ if and only if $\sum_{j=0}^{k-1} f(x_j)=0$ 
whenever $(x_j)_{j=0}^{k-1}$ is a closed orbit for \emph{some}
sequence of admissible maps $T_{\omega_1},\dots,T_{\omega_{k}}$. Unfortunately,
to use such a criterion it may be necessary to check a very large
number of trajectories. Yet, if $\wp\in{]0,1[}$, then the situation may be much
simpler. Indeed, if $\Sigma^2_\wp=0$, then it must be $f=g-g\circ T_0=g-g\circ T_1$ and hence $g\circ T_1 \circ T_0^{-1}=g$. Thus if $T_1
\circ T_0^{-1}$ is ergodic,\footnote{Note that this may easily fail
even if $T_0\neq T_1$. Indeed, consider the case
\[
A_0=\begin{pmatrix}
1&1\\2&3
\end{pmatrix}\quad \quad A_1=\begin{pmatrix}
1&1\\1&2
\end{pmatrix}
\]
Then $T_1(T_0^{-1}(x_1,x_2))=(x_1,x_2-x_1)$. In fact, all the functions
$g(x_1,x_2)=\tilde g(x_1)$ are invariant. The identity $f=g-g\circ
T_0=g-g\circ T_1$ yields $\sum_{k=0}^{N-1} f\circ\pi\circ
F^k(\omega,x)=g(x)-g\circ T_i\circ T_{\omega_{N}}\circ\dots\circ T_{\omega_1}(x)$ for $i\in\{0,1\}$ and hence $\Sigma^2_\wp=0$ if
$g$ is in $L^2$.} 
then $g$ must be constant and hence $f\equiv 0$ contrary to assumptions. That is: if $\{T_0,T_1\}$ are admissible and $T_1\circ T_0^{-1}$ is ergodic, then $\Var_\wp>0$.
\end{rem}
\begin{rem}
Note that one cannot possibly extend our results to include all sequences; for instance, there exist sequences containing alternating, ``deterministic", stretches of either $T_0$'s or $T_1$'s. If these stretches are of rapidly and ever increasing length, then the variance fails to exist.
\end{rem}
Before proving such a strong result we will obtain its averaged (annealed) version.
\begin{lem}\label{lem:clt-averaged}
For each $\wp\in[0,1]$ there exists $\Var_\wp\in\bR_+$ such that it holds
\[
N^{-\frac 12}S_N\Rightarrow
\cN\bigl(0, \Var_\wp\bigr)\quad under\  \bfP_\wp.
\]
\end{lem}

In turn such a result is based on a fine understanding of the dynamical
properties of certain transfer operators associated to the process. 

\begin{rem} 
Note that one could obtain similar results for any finite collection of smooth symplectic hyperbolic maps in any dimension or piecewise smooth maps in dimension two. This can be achieved at the price of using in the following section the functional setting of \cite{GL, BaT} or \cite{DL} for the piecewise smooth case.
\end{rem}
Our first task will be to obtain some information on the spectral
properties of such operators. To do so in a useful way it is necessary to
introduce appropriate functional spaces. Instead of appealing to the
general theory developed in \cite{BKL, Ba1, FR, BaT, GL} we will take
advantage of the simplicity of the present setting and introduce
explicitly a particularly simple version of such a theory. We will then see
how it can be used to address the ergodic theoretical questions we are
interested in.

\section{Spectral properties of the Transfer operators} \label{sec:gap}

For further use (see section \ref{sec:quenched}) we need to
study more general automorphisms than the one introduced in the previous
section, namely $\bigoplus_{1}^d T_i\;:\;\bT^{2d}\to\bT^{2d}$ defined by
$\bigoplus_{1}^d T_ix=\bigoplus_{1}^d A_i\mod 1$ where $\bigoplus_{1}^d
A_i$
is the $d$-fold direct sums of the matrices $A_0,A_1\in SL(2,\bZ)$, that
is
\[
\bigoplus_{1}^d A_i =
\begin{pmatrix}
A_i &            & \\
   & \ddots & \\
   &            & A_i \\
\end{pmatrix}
\in SL(2d,\bZ),
\]
with $d\in\{1,2\}$. These matrices are symplectic with respect to the
symplectic form
\[
J_{2d} =
\begin{pmatrix}
J_2 &            & \\
   & \ddots & \\
   &            & J_2 \\
\end{pmatrix},
\]
where $J_2 = \bigl( \begin{smallmatrix} 0 & 1 \\ -1 & 0\\
\end{smallmatrix} \bigr)$ is the standard symplectic form in two
dimensions.\footnote{In fact all the following can be easily generalized
to any set of symplectic toral automorphisms which preserves the standard
sector in any dimension, see \cite{LW1, LW2} for more details on the
necessary machinery.}

Let us introduce the transfer operators $\cL_{T_i}^{(d)} \varphi := \varphi\circ (\bigoplus_1^d T_i)^{-1}$ induced by the above automorphisms. We will consider the operator obtained from
the latter by averaging over the Bernoulli measure, namely 
\[
\cL^{(d)}_{\wp}:=\wp \,\cL^{(d)}_{T_0}+(1-\wp) \cL^{(d)}_{T_1}.
\]
We also need to study perturbed operators of the form
\[
\cL^{(d)}_{g,\wp} \varphi:=\cL^{(d)}_{\wp} (e^g \varphi).
\]
In order to avoid unnecessary proliferation of indices, we set
\begin{equation}\label{eq:transfer1D}
\cL_{T_i}:=\cL^{(1)}_{T_i}, \quad \cL_{\wp}:=\cL^{(1)}_{\wp}, \quad\text{and}\quad \cL_{g,\wp} :=\cL^{(1)}_{g,\wp}.
\end{equation}
Finally, notice that the transfer operator $\cL_{\wp}$ and the Markov operator $Q_\wp$ are dual:
\begin{equation}\label{eq:dual}
\Leb(f Q_\wp g) = \Leb(\cL_\wp f \cdot g).
\end{equation}

To study such operators it is necessary to introduce appropriate
Banach spaces (see \cite{GL,BaT,FR}). Here, given the simplicity of
the situation, we can quickly introduce and use spaces inspired by
\cite{BaT,FR} whereby making the presentation self-consistent.\footnote{Actually our
choice is more flexible than the one in \cite{FR}, in the spirit of
\cite{BaT}, and would allow to treat $\cC^k$ maps, although it is not the goal here.}

Given $v:=(v_1,\dots,v_{2d})\in\bR^{2d}$, let us denote
$\check{v}:=(v_1,v_3,\dots,v_{2d-1})\in\bR^d$ and
$\hat{v}:=(v_2,v_4,\dots,v_{2d})$. Then the standard inner product in
$\bR^d$ of the latter two reads $\langle\check{v}, \hat{v}\rangle =
\sum_{k=0}^{d-1} v_{2k+1} v_{2(k+1)}$.
Consider now the cones $\cC_-:=\{v\in\bR^{2d}\;:\; \langle\check{v}, 
\hat{v}\rangle \leq
0\}$, $\cC_+:=\{v\in\bR^{2d}\;:\; \langle\check{v},\hat{v}\rangle \geq
0\}$. Then
there exists $1<\lambda\leq\Lambda$ and $C_0$, depending only on
$\{A_0,A_1\}$,
such that, for all $v\in\cC_-$, $(i_1,\dots,i_n)\in \{0,1\}^n$, and
$n\in\bN$,
\begin{equation}\label{eq:contraction}
C_0^{-1}\lambda^n\leq 
\frac{\|(\bigoplus_1^d A_{i_1}\cdots \bigoplus_1^dA_{i_n})^{-1} v\|}{\|v\|},
\frac{\|(\bigoplus_1^d A_{i_1}^T \cdots \bigoplus_1^d A_{i_n}^T)^{-1} v\|}{\|v\|}
\leq C_0\Lambda^n;
\end{equation}
and, for all $v\in\cC_+$, $(i_1,\dots,i_n)\in \{0,1\}^n$, and $n\in\bN$,
\begin{equation}\label{eq:expansion}
C_0^{-1}\lambda^n\leq 
\frac{\|(\bigoplus_1^d A_{i_1}\cdots \bigoplus_1^dA_{i_n})v\|}{\|v\|},
\frac{\|(\bigoplus_1^d A_{i_1}^T \cdots \bigoplus_1^d A_{i_n}^T) v\|}{\|v\|}
\leq C_0\Lambda^n.
\end{equation}

Moreover one can compute that there exists $\beta>1$ such that
$(\bigoplus_1^d A_i^{-1})\cC_-\subset \cC_\beta$, where
$\cC_\beta:=\{v\in\bR^{2d}\;:\; \beta^{-1}\|\check{v}\|^2\leq -\langle
\check{v},\hat v\rangle\leq \beta\|\check{v}\|^2\}\subset \text{int
}\cC_-$.

Now, we proceed to define the norm for the Banach space we want to
consider. Notice that the natural objects in these cones are not vectors
but Lagrangian subspaces. Recall that, given a symplectic form $J$, a
Lagrangian subspace $E\subset \bR^{2d}$ is a $d$-dimensional subspace such
that $\langle v, J w\rangle = 0$ for all $v,w \in E$.

For our choice of symplectic form, every Lagrangian subspace can also
be written as the set $E=\{ v\;:\; \hat{v} = -U \check{v}\}$ for a specific
symmetric $d\times d$ matrix $U$. Our convention is to write a minus sign in front of the $U$ here, because then $E\subset \cC_\beta$ if and only if $\beta^{-1}\Id \leq U \leq \beta \Id$.

Let us denote the set of Lagrangian subspaces as $\bL$.
For a Lagrangian subspace $E$ and a vector $k$, we set $\langle
E,k\rangle:=\displaystyle{\sup_{\substack{v\in E\\ \|v\|=1}}}\langle
v,k\rangle$.
Then, for each $p,q\in\bR_+$ and $f\in\cC^\infty(\bR^{2d},\bR)$ we
define the norm
\[
\|f\|_{p,q}:=\sup_{\substack{E \in \bL\\ E \subset
\cC_-}}\sum_{k\in\bZ^{2d}\backslash
\{0\}}|f_k|\frac{|k|^p}{1+|\langle E,k\rangle|^{p+q}}+|f_0|,
\]
where $f_k$ are the Fourier coefficients of $f$.

Notice that the $n$th power of $\cL^{(d)}_\wp$ can be expanded
\[
[\cL^{(d)}_\wp]^n = \sum_{j=1}^n\sum_{i_j=0}^1
\wp^{\delta_{i_1,0}+\dots+\delta_{i_n,0}}
(1-\wp)^{\delta_{i_1,1}+\dots+\delta_{i_n,1}} \, \cL^{(d)}_{T_{i_n}} \dots \cL^{(d)}_{T_{i_1}}.
\]
Because the Bernoulli weights above sum to unity,
\[
\|[\cL^{(d)}_\wp]^n f\|_{p,q} \leq \sup_{(i_1,\dots,i_n)\in \{0,1\}^n}
\|\cL^{(d)}_{T_{i_n}} \dots \cL^{(d)}_{T_{i_1}} f\|_{p,q}.
\]

A straightforward computation shows that $(\cL^{(d)}_{T_i}
f)_k=f_{(\bigoplus_1^d A_i^T) \,k}$ for $k\in \bigoplus_1^d\bZ^2 \cong
\bZ^{2d}$.
Hence, for each $n\in\bN$ and $(i_1,\dots,i_n)\in \{0,1\}^n$,
\[
\| \cL^{(d)}_{T_{i_n}} \dots \cL^{(d)}_{T_{i_1}} f  \|_{p,q}=\sup_{\substack{E \in \bL\\ E
\subset \cC_-}}\sum_{k\in\bZ^{2d}\backslash
\{0\}}|f_k|\frac{|(\bigoplus_1^d A_{i_1}^T\dots \bigoplus_1^d
A_{i_n}^T)^{-1}k |^p}{1+ |\langle E,(\bigoplus_1^d
A_{i_1}^T\dots \bigoplus_1^d A_{i_n}^T)^{-1}k  \rangle
|^{p+q}}+|f_0|.
\]

We begin estimating the summand. Before that, we simplify notation and
rename the matrix product. To avoid the problem of too many indices, we
denote it only with the subscript $n$ and assume the dimension and the
sequence to be implicit:
\[
\mfa_n := \bigoplus_1^dA_{i_1}^T\cdots \bigoplus_1^dA_{i_n}^T.
\] 

Using \eqref{eq:contraction} and the fact $E \subset \cC_-$, we have
\begin{equation}\label{eq:angle0}
C_0\Lambda^n \langle  E_n,k \rangle\geq\langle E,\mfa_n^{-1} k \rangle
\geq C_0^{-1} \lambda^n \langle E_n,k \rangle \quad \text{ where } \quad
E_n:=(\mfa_n^{-1})^T E
\end{equation}
and consequently,
\[
|f_k|\frac{|\mfa_n^{-1}k |^p}{1+ |\langle E,\mfa_n^{-1}k  \rangle
|^{p+q}} \leq  C_0^{p+q} |f_k|\frac{|\mfa_n^{-1}k |^p}{1+ \lambda^{n(p+q)}
|\langle E_n,k  \rangle
|^{p+q}}.
\]

Now there are two possible cases depending on where $\mfa_{n}^{-1}k$ lies.
If $\mfa_{n-1}^{-1}k \notin \cC_-$, then $|\mfa_n^{-1} k| \leq C_0
\Lambda\lambda^{-n+1} |k|$ and we have
\[
\frac{|\mfa_n^{-1}k |^p}{1+ \lambda^{n(p+q)} |\langle E_n,k  \rangle
|^{p+q}} \leq C_0^p(\Lambda\lambda^{-1})^p \lambda^{-np} \frac{|k|^p}{1+
|\langle E_n,k  \rangle |^{p+q}}.
\]

While if $\mfa_{n-1}^{-1}k \in \cC_-$, then
\begin{equation}\label{eq:angle}
\langle E_n, k \rangle\geq
\frac{\lambda \Lambda^{-2n+1}|k|}{2 C_0^3 \beta\sqrt{1+\beta^2}}=:\frac{|k|}{B_{\beta,n}}.
\end{equation}
Indeed, setting $k_n:= \mfa_{n-1}^{-1}k$,
\[
\begin{split}
\langle   \bigl(\bigoplus_1^dA_{i_n}^{-1}\bigr) E,  k_n \rangle&\geq 
\inf_{\substack{E\in\bL\\E\subset \cC_\beta}}\langle E, k_n \rangle
=\inf_{\beta^{-1}\Id\leq U\leq \beta\Id}\
\sup_{\check{v}\in\bR^d}\frac{\langle \check{v},\check{k}_n\rangle-\langle
\check{v},U\hat k_n\rangle}{\sqrt{|\check{v}|^2+|U\check{v}|^2}}\\
&\geq \inf_{\beta^{-1}\Id\leq U\leq \beta\Id}\frac{-\langle
\check{k}_n,\hat k_n\rangle+\langle \hat k_n,U\hat k_n\rangle}{\sqrt{|\hat
k_n|^2+|U\hat k_n|^2}}\\
&\geq \frac{|\hat k_n|}{\beta\sqrt{1+\beta^2}},
\end{split}
\]
where we have chosen $\check{v}=-\hat k_n$ in the second line. On the
other hand the choice $\check{v}=U^{-1}\check{k}_n$ yields
\[
\langle  \bigl(\bigoplus_1^dA_{i_n}^{-1}\bigr) E,  k_n \rangle\geq 
\inf_{\substack{E\in\bL\\E\subset \cC_\beta}}\langle E, k_n \rangle\geq 
\frac{|\check{k}_n|}{\beta\sqrt{1+\beta^2}},
\]
The inequality \eqref{eq:angle} follows from the above estimates,  \eqref{eq:angle0} and $C_0\Lambda^{n-1} |k_n|\geq |k|$.

Next, we consider two subcases. If $|k| \geq B_{\beta,n}$, then $|\langle
E_n,k \rangle |\geq 1$. Hence,
\[
\frac{|\mfa_n^{-1}k |^p}{1+ \lambda^{n(p+q)} |\langle E_n,k  \rangle
|^{p+q}} \leq C_0^p \frac{\Lambda^{pn}\lambda^{-(p+q)n} |k |^p}{\bigl
|\bigl\langle E_n,k \bigr \rangle \bigr |^{p+q}}
\leq 2C_0^p \frac{\Lambda^{pn}\lambda^{-(p+q)n} |k |^p}{1+\bigl
|\bigl\langle E_n,k \bigr \rangle \bigr |^{p+q}},
\]
which is a good estimate provided $q$ is large enough so that
$\Lambda^{p}\lambda^{-(p+q)} < 1$. The remainder is a finite sum which can
be estimated because if $E \subset \cC_-$, then $E_n \subset
(\mfa_n^{-1})^T \cC_- \subset \cC_-$. Thus,  we have
\[
\sup_{\substack{E \in \bL\\ E \subset \cC_-}}
\sum_{\substack{k\in\bZ^{2d}\backslash
\{0\}\\ |k|<B_{\beta,n}}}\frac{ |f_k|\,|\mfa_n^{-1}k |^p}{1+
\lambda^{n(p+q)} |\langle E_n,k  \rangle |^{p+q}} \leq
C_0^p \Lambda^{np} \sup_{\substack{E \in \bL\\ E \subset \cC_-}}
\sum_{\substack{k\in\bZ^{2d}\backslash
\{0\}\\ |k| < B_{\beta,n}}} \frac{|f_k|\,|k |^p}{1+ |\langle E,k  \rangle
|^{p+q}}.
\]
Accordingly, setting $\tilde \mu:=
\max\{\lambda^{-p},\Lambda^p \lambda^{-p-q}\}$ we can
collect all the above inequalities as
\begin{equation}
\label{eq:lasota-0}
\begin{split}
\|[\cL^{(d)}_\wp]^n f\|_{p,q}&\leq C_1\|f\|_{p,q},\\
\|[\cL^{(d)}_\wp]^n f\|_{p,q}&\leq 2C_0^{q+2p}\tilde \mu^n\|f\|_{p,q}+
 \sup_{\substack{E \in \bL\\ E \subset \cC_-}}
\sum_{\substack{k\in\bZ^{2d}\backslash
\{0\}\\ |k| < B_{\beta,n}}} \frac{C_0^{q+2p} \Lambda^{np} |f_k|\, |k
|^p}{1+ |\langle E,k  \rangle |^{p+q}} +|f_0|\\
&\leq C_2\tilde \mu^n\|f\|_{p,q}+B_n\|f\|_{p-1,q+1},
\end{split}
\end{equation}
where $B_n= C_0^{q+2p}B_{\beta,n} \Lambda^{np}$. Next, for each $\mu\in
(\tilde
\mu,1)$ choose $n_0$ such that $C_2\tilde \mu^{n_0}\leq \mu^{n_0}$
and, for
each $n\in\bN$, write $n=kn_0+m$ with $m\in\{0,\dots,n_0-1\}$. One can
thus iterate the second of the \eqref{eq:lasota-0} and obtain
\[
\|[\cL^{(d)}_\wp]^n f\|_{p,q}\leq
\mu^{kn_0}\|\cL^mf\|_{p,q}+B_{n_0}\sum_{j=0}^{k-1}\mu^{jn_0}\|\cL^{(k-1-j)n_0+m}f\|_{p-1,q+1},
\]
which finally yields
\begin{equation}
\label{eq:lasota-1}
\begin{split}
\|[\cL^{(d)}_\wp]^n f\|_{p,q}&\leq C_1\|f\|_{p,q},\\
\|[\cL^{(d)}_\wp]^n f\|_{p,q}&\leq C_3\mu^n\|f\|_{p,q}+B\|f\|_{p-1,q+1},
\end{split}
\end{equation}
with $B=C_2B_{n_0}(1-\mu^{n_0})^{-1}$.
We can then consider the closure, $\cB^{p,q}$, of $\cC^\infty$ in the
space of
distributions with respect to the norms $\|\cdot\|_{p,q}$. It is easy to
prove the following:

\begin{lem}
\label{lem:compactness}
The operators $\cL^{(d)}_\wp$ are well defined bounded operators on
$\cB^{p,q}$, provided $\Lambda^p<\lambda^{p+q}$. In addition, the unit
ball of $\cB^{p,q}$ is relatively compact in $\cB^{p-1,q+1}$.
\end{lem}

\begin{thm}
\label{thm:quasi-compactness}
If $\Lambda^p<\lambda^{p+q}$, the operator $\cL^{(d)}_\wp$ acting on $\cB^{p,q}$ has an essential
spectral radius smaller than $\mu$. The rest of
the spectrum consists of finitely many eigenvalues of finite multiplicity, all in the unit disk.
The only eigenvalue of modulus one is one and the constant function equal to one is the
corresponding eigenfunction.
\end{thm}
\begin{proof}
Lemma \ref{lem:compactness}, the Lasota-Yorke type inequalities
\eqref{eq:lasota-1} imply the result (see \cite{Babook, BKL} for more
details).
\end{proof}

Before proceeding, let us mention that for any $r,n\in\bN$ we endow the space $\cC^{r}(\bT^n,\bR)$ with the norm $\|g\|_{\cC^r} := \sum_{s=0}^r \| g^{(s)} \|_\infty$.

Moreover, a simple computation shows that $\cC^r\subset \cB^{p,q}$, provided $r>p+2d$. 

\begin{cor}\label{cor:trivial}
The equation \eqref{eq:mixing} holds true.
\end{cor}
\begin{proof}
First of all notice that, for all $f\in\cB^{p,q}$ and $g\in\cC^q$ holds
\[
\begin{split}
|\Leb(fg)|&\leq\sum_{l\in\bZ^{2}}|f_l|\,|g_{-l}|\leq \|g\|_{\cC^q} \left(|f_0|+\sum_{l\in\bZ^{2}\setminus\{0\}}|f_l|\,|l|^{-q}\right)\\
&\leq \|g\|_{\cC^q}\|f\|_{p,q}\max\left(1,\sup_{\substack{E\in\bL\\ l\in\bZ^2\setminus\{0\}}}\frac{1+|\langle E,l\rangle|^{p+q}}{|l|^{p+q}}\right)
\leq 2\|g\|_{\cC^q}\|f\|_{p,q}.
\end{split}
\]
If $p,q$ satisfy the hypothesis of Theorem \ref{thm:quasi-compactness},
$\cL_\wp$ has a spectral gap $\delta_\wp>0$. Thus,
\[
|\Leb(f Q_\wp^ng)-\Leb(f)\Leb(g)|=|\Leb(\cL_\wp^nf \cdot
g)-\Leb(f)\Leb(g)|\leq C(1-\delta_\wp)^n\|f\|_{p,q}\|g\|_{\cC^{q}},
\]
because decomposing $\cL_\wp := \cQ+\cR$ with $\cQ f:=\Leb(f)$ we have $\cQ\cR=\cR\cQ=0$ and therefore $\cL_\wp^n f = \cR^n f + \cQ f$, where $\| R^n \|_{\cL(\cB^{p,q})}\leq C(1-\delta_\wp)^n$ by the spectral radius formula. 
\end{proof}
Here is the last fact we need to know about the above functional analytic
setting.
\begin{lem}\label{lem:multiplication}
For each function $g\in\cC^{2p+q+2d+1}(\bT^{2d},\bR)$ the multiplication
operator $M_g$ defined by $M_gf=gf$, is bounded in $\cB^{p,q}$ by $C\|g\|_{\cC^{2p+q+2d+1}}$.
\end{lem}
\begin{proof}
We define the norm\footnote{We are aware that our choices of norms and the subsequent estimates, are not the optimal ones. We are simply trying to simplify the arguments as much as possible even at the expense of some, not really relevant, optimality.}
\[
\| g \|_{r} := \sup_{k \in \bZ^{2d}} |g_k| (1+|k|^{r}).
\]
Clearly $g\in\cC^{r}(\bT^{2d},\bR)$ implies $\|g\|_{r}\leq \|g\|_{\cC^r}$.
The $\cB^{p,q}$-norm of the product then reads
\begin{equation}\label{eq:prod-norm}
\| fg \|_{p,q} = \sup_{\substack{E \in \bL\\ E \subset \cC_-}}
\sum_{\substack{k,l\in\bZ^{2d}\\ k \neq 0}} |f_l| |g_{k-l}| \frac{ |k
|^p}{1+  |\langle E,k  \rangle |^{p+q}} + \sum_{l \in \bZ^{2d}} |f_l|
|g_{-l}|.
\end{equation}
Let us analyze the second term first:
\[
\sum_{l \in \bZ^{2d}} |f_l| |g_{-l}| \leq \| g\|_{r} \left( |f_0|
+\sum_{l \in \bZ^{2d} \setminus \{0\}} |f_l| |l|^{-r} \right).
\]
The desired bound follows, if $r\geq q$, from
\[
\frac{1}{2}| l |^{-q} \leq \frac{ | l |^p}{1+  |\langle E,l  \rangle
|^{p+q}}.
\]
Now, look at the summand in the first term of \eqref{eq:prod-norm}, ignoring
the $l=0$ case that can be taken care of separately.
\[
\begin{split}
|f_l| |g_{k-l}| \frac{ |k |^p}{1+  |\langle E,k  \rangle |^{p+q}} &\leq\|g\|_r
\left[ |f_l|  \frac{ |l |^p}{1+  |\langle E,l  \rangle |^{p+q}} \right]  \\
&\quad \times \frac{ |k |^p}{1+  |\langle E,k  \rangle |^{p+q}} \frac{1+
|\langle E,l  \rangle |^{p+q}}{ |l |^p} \frac{1}{1+|k-l|^{r}}.
\end{split}
\]
Notice that $|\langle E, l\rangle|\leq |\langle E, k\rangle|+|\langle E, k-l\rangle|$. Thus, on the one hand
\[
\sum_k \frac{ |k |^p}{1+  |\langle E,k  \rangle |^{p+q}} \frac{1+  |\langle E,k
\rangle |^{p+q}}{ |l |^p} \frac{1}{1+|k-l|^{r}} \leq \sum_k\frac{ |k |^p } {|l |^p(1+|k-l|^{r})}\leq C.
\]
One the other hand
\[
\sum_k \frac{ |k |^p}{1+  |\langle E,k  \rangle |^{p+q}} \frac{1+  |\langle E,k-l
\rangle |^{p+q}}{ |l |^p} \frac{1}{1+|k-l|^{r}} \leq 3\sum_k\frac{ |k |^p } {|l |^p(1+|k-l|^{r-p-q})}\leq C,
\]
provided $r> 2p+q+2d$. The general term of $(|\langle E, k\rangle|+|\langle E, k-l\rangle|)^{p+q}$ is bounded similarly, using $|\langle E, k\rangle|^n |\langle E, k-l\rangle|^{p+q-n} \leq (1+|\langle E, k\rangle|^n)(1+|\langle E, k-l\rangle|^{p+q-n})$ for each $n=0,\dots,p+q$.
 \end{proof}

\section{Averaged CLT}\label{sec:average}
To establish Lemma \ref{lem:clt-averaged} it suffices to compute
\[
\lim_{N\to\infty}\bE_{{\bf P}_\wp}\left(e^{i\frac \lambda{\sqrt
N}\sum_{k=0}^{N-1}
f_{\omega,k}}\right),
\]
where $f_{\omega,k} := X_k(\omega,\cdot) := f\circ
T_{\omega_k}\circ\dots\circ T_{\omega_1}$ with $f_{\omega,0}:=f$,
and show that this limit is the characteristic function $e^{-\frac
12\lambda^2\Var_\wp}$ of the centered normal distribution with some
variance $\Var_\wp$. Recalling the transfer operators in \eqref{eq:transfer1D},
\begin{equation}\label{eq:powers}
\bE_{{\bf P}_\wp}\left(e^{i\frac \lambda{\sqrt N}\sum_{k=0}^{N-1}
f_{\omega,k}}\right)=
\Leb(\cL_{i\lambda N^{-\frac 12} f,\wp}^N\, 1).
\end{equation}
Hence Theorem \ref{thm:quasi-compactness} and Lemma
\ref{lem:multiplication} show that $\cL_{i\lambda N^{-\frac 12} f,\wp}$ is
a bounded operator on $\cB^{p,q}$ provided $f\in\cC^{2p+q+5}$ and depends
analytically on $\lambda$. To continue it is necessary to study the leading
eigenvalue of such an operator.

Note that, in general, given any positive operator $\cL$ on the spaces $\cB^{p,q}$ with maximal
simple eigenvalue one, with a spectral gap and $\Leb(\cL\vf)=\Leb(\vf)$ for
each smooth $\vf$, for any smooth complex valued function $g$ we can
define the family of operators $\cL_\nu\vf:=\cL(e^{\nu g}\vf)$ and, thanks to Lemma \ref{lem:multiplication}, the
standard perturbation theory applies. Thus there exists $\phi_\nu$,
$\mu_\nu$, with $\mu_0=1$, such that 
\[
\cL_\nu\phi_\nu=\mu_\nu \phi_\nu,\quad \Leb( \phi_\nu)=1.
\]
Differentiating this relation with respect to $\nu$ and integrating
one readily obtains
\begin{equation}\label{eq:oneder}
\mu'_\nu=\Leb \left( ge^{\nu g}\phi_\nu+e^{\nu g}\phi'_\nu \right)
\end{equation}
and, setting $\nu=0$, $\phi'_0=(\Id-\cL)^{-1}[\cL(g\phi_0)-\phi_0\Leb(
g\phi_0)]=\sum_{n=0}^\infty\cL^n[\cL(g\phi_0)-\phi_0\Leb(
g\phi_0)]$.\footnote{The latter is well defined since $(\Id-\cL)^{-1}$ is
applied on a function from which the eigendirection of $\cL$ corresponding to eigenvalue 1 has been projected out, and because of the spectral gap.}
Finally, differentiating again yields
\begin{equation}\label{eq:twoder}
\mu''_0=\Leb( g^2\phi_0)+2\sum_{n=0}^\infty\Leb( g \cL^n \left[ \cL(g\phi_0)- \phi_0\Leb(
g\phi_0)\right]). 
\end{equation}

Thus, by standard perturbation theory and in view of Theorem \ref{thm:quasi-compactness} we can write
$\cL_\nu=\mu_\nu \cQ_\nu+\cR_\nu$ where $\cQ_\nu^2=\cQ_\nu$, $\cR_\nu\cQ_\nu=\cQ_\nu\cR_\nu=0$, the spectral radius of $\cR_\nu$ is smaller than $\rho<1$ for all $|\nu|\leq \nu_0$ for some $\nu_0>0$, and $|\mu_\nu-1-\mu'_0\nu-\frac 12 \mu''_0\nu^2|\leq C|\nu|^3$ for some fixed constant $C>0$ and $|\nu|\leq \nu_0$. In addition, $\cQ_\nu$ is a rank one operator of the form $\phi_\nu\otimes m_\nu$ where $m_\nu$ belongs to the dual of the space, $m_0=\Leb$, and $|m_\nu(1)-1|\leq C|\nu|$.

If we apply the above to the operator $\cL_{i\lambda N^{-\frac 12} f,\wp}$, $\nu=i\lambda N^{-\frac 12}$ (hence $g=f$, and $\phi_0\equiv 1$), then remembering equation \eqref{eq:powers}
it follows that 
\begin{equation} \label{eq:charfn}
\begin{split}
\bE_{{\bf P}_\wp}\left( e^{i\frac \lambda{\sqrt
N}\sum_{k=0}^{N-1} f_{\omega,k}}\right)&=\Leb\left(\mu_\nu^N\cQ_\nu 1+\cR_\nu^N1\right)=e^{N\ln(\mu_\nu)}(1+\cO(|\lambda|N^{-\frac 12}))+\cO(\rho^N)\\
&=e^{-\frac 12\mu''_0\lambda^2+N\cO\left((|\lambda|/N^{-\frac 12})^3\right)}+\cO(|\lambda|N^{-\frac 12}+\rho^N),
\end{split}
\end{equation}
because $\Leb(\cQ_\nu 1)=\Leb(\phi_\nu m_\nu(1))=m_\nu(1)$ and $|\Leb(\cR_\nu^N 1)|\leq \| \cR_\nu^N \|_{\cL(\cB^{p,q})}\leq C\rho^N$.
Hence, if $|\lambda|/\sqrt N$ is sufficiently small, we have
\begin{equation}\label{eq:clt}
\left|\bE_{{\bf P}_\wp}\left( e^{i\frac \lambda{\sqrt
N}\sum_{k=0}^{N-1} f_{\omega,k}}\right)-e^{-\frac 12\mu''_0\lambda^2}\right|\leq C \frac{1+|\lambda|^3}{\sqrt  N}.
\end{equation}

\begin{lem}\label{lem:variance}
The quantity $\Sigma_\wp^2\in\bR$ defined by
\[
\Sigma_\wp^2 :=\lim_{N\to\infty}\frac
1N\bE_{{\bf P}_\wp}\left(\left[\sum_{k=0}^{N-1} X_k\right]^2\right)
\]
is always nonnegative and given by 
\begin{equation}\label{eq:var-series}
\Leb(f^2)+2\sum_{n=1}^{\infty}\Leb( f \cL_\wp^n f) = \mu''_0.
\end{equation}
The map $[0,1]\to\bR_+:\wp\mapsto\Var_\wp$ is analytic.
Moreover, if $f$ is not a $\cC^0$ simultaneous coboundary for the admissible automorphisms $T_i$ (see Remark \ref{livsic}), then $\Sigma_\wp^2>0$.
\end{lem}
\begin{proof}
A direct computation yields 
\[
\begin{split}
\bE_{{\bf P}_\wp}\left(\left[\sum_{k=0}^{N-1} X_k\right]^2\right)&=\sum_{k=0}^{N-1}\bE_{{\bf P}_\wp}(X_k^2)+2\sum_{0 \leq j < k \leq N-1}\bE_{{\bf P}_\wp}(X_kX_j)\\
&=N\Leb(f^2)+2\sum_{0\leq j<k \leq N-1}\Leb(fQ_\wp^{k-j}f)\\
&=N\left[\Leb(f^2)+2\sum_{n=1}^{N-1}\Leb( f \cL_\wp^n f)\right]-2\sum_{n=1}^{N-1}n\,\Leb( f \cL_\wp^n f)
\end{split}
\]
Using Corollary \ref{cor:trivial}, the last sum converges exponentially fast in $n$. Hence $\Sigma_\wp^2$ exists and is nonnegative simply because it is the limit of a nonnegative quantity. To address this last issue, suppose $\Sigma_\wp^2=0$. Then
\[
\left|\bE_{{\bf P}_\wp}\left(\left[\sum_{k=0}^{N-1} X_k\right]^2\right)\right|\leq 2N\sum_{n=N}^{\infty}|\Leb( f \cL^n_\wp f)|+2\sum_{n=1}^{N-1}n\,|\Leb( f \cL^n_\wp f)|\leq C
\]
uniformly in $N$. This means that the random variables
$Z_N:=\sum_{k=0}^{N-1} X_k$ are uniformly bounded in $L^2$. By the Banach--Alaoglu Theorem, they
form a weak-* relatively compact set. 
We can then extract a subsequence $(N_j)_{j=1}^\infty$ such that, for each $\vf\in L^2(\Omega,{\bf P}_\wp)$,
\[
\lim_{j\to\infty}\bE_{{\bf P}_\wp}(\vf Z_{N_j})=\bE_{{\bf P}_\wp}(\vf Y)
\]
for some $L^2$ random variable $Y$. If we choose $\vf$ to be a function of the $x$ only, it follows
that 
\[
\lim_{j\to\infty}\Leb(\vf\sum_{n=0}^{N_j-1} Q_\wp^nf)=\Leb(\vf g)
\]
where $g=\bE_\wp(Y)\in L^2(\bT^2,\Leb)$. On the other hand, for each smooth $\vf$,
\[
\begin{split}
\Leb(\vf(f-g+Q_\wp g))&=\lim_{j\to\infty}\Leb\left(\vf\left(f-\sum_{n=0}^{N_j-1} Q_\wp^nf+\sum_{n=0}^{N_j-1} Q_\wp^{n+1}f\right)\right)\\
&=\lim_{j\to\infty}\Leb(\vf Q_\wp^{N_j} f)=0.
\end{split}
\] 
That is $f=g-Q_\wp g$. Next, consider the $L^2$ random variables $G_n:=g\circ \pi \circ F^n$
and $M_{n+1}=\sum_{k=0}^{n}( X_k+G_{k+1}-G_k)=G_{n+1}-G_0+\sum_{k=0}^{n}X_k$. For each $N\in\bN$, we use Jensen's inequality to get
\begin{equation}\label{eq:upper}
\begin{split}
C &\geq \bE_{{\bf P}_\wp}\left(\left[\sum_{k=0}^{N-1} X_k\right]^2\right)=\bE_{{\bf P}_\wp}\left(\left[M_N-G_N+g\right]^2\right)\\
&\geq \bE_{{\bf P}_\wp}\left(M_N^2\right)-2\sqrt{\bE_{{\bf P}_\wp}\left(M_N^2\right)\bE_{{\bf P}_\wp}\left([G_N-g]^2\right)}.
\end{split}
\end{equation}
In fact, the process $(M_n)$ is a martingale, since
\[
\bE_{{\bf P}_\wp}(G_{k+1}\;|\; x,\omega_1,\dots,\omega_k)=\bE_{{ P}_\Leb}(g(x_{k+1})\;|\; x_k)=Q_\wp g(x_k)=(Q_\wp g)\circ\pi\circ F^k(x,\omega).
\]
Thus,
\[
\begin{split}
\bE_{{\bf P}_\wp}\left(M_N^2\right)&=\sum_{k=0}^{N-1}\bE_{{\bf P}_\wp}\left(\left[X_k+G_{k+1}-G_k\right]^2\right)\\
&=N\left\{\wp\;\Leb([f+g\circ T_0-g]^2)+(1-\wp)\Leb([f+g\circ T_1-g]^2)\right\}.
\end{split}
\]
The inequality \eqref{eq:upper} and the boundedness of $\bE_{{\bf
P}_\wp}\left([G_N-g]^2\right)$ imply that $\wp\;\Leb([f+g\circ
T_0-g]^2)+(1-\wp)\Leb([f+g\circ T_1-g]^2)=0$, that is $f+g\circ T_i-g=0$ for each admissible $T_i$.

The continuity of $g$ follows from the usual Livschitz rigidity arguments.
\footnote{In fact, in the present simple case one can provide the following direct proof: clearly $g=(\Id-\cL_{T_i})^{-1}\cL_{T_i} f=\sum_{k=1}^{\infty}\cL_{T_i}^kf$ for an admissible choice of $T_i$, convergence taking place in the $\|\cdot\|_{p,q}$ norm. Let $v^{u,s}$ be the unstable and stable vectors of $T_i$, respectively, and $\varphi\in\cC^\infty$. Then
\[
|\Leb(\langle v^u,\nabla\vf\rangle g)|\leq\sum_{k=1}^\infty|\Leb(\vf \langle v^u,\nabla \cL_{T_i}f\rangle)|
=\sum_{k=1}^\infty|\Leb(\vf \langle T_i^{-k}v^u, \cL_{T_i}^k \nabla f\rangle)|
\leq \sum_{k=1}^\infty \|\nabla f\|_\infty\lambda^{-k}\|\vf\|_{L^1}.
\]
On the other hand, $g(x)=\sum_{k=0}^nf\circ T_i^k+g\circ T_i^{n+1}$, and the mixing of $T_i$ (proven exactly as in Corollary \ref{cor:trivial}) implies
\[
|\Leb(\langle v^s,\nabla\vf\rangle g)|\leq \sum_{k=0}^{\infty}|\Leb(\langle v^s,\nabla\vf\rangle f\circ T^k_i)|
\leq  \sum_{k=0}^{\infty}\|\nabla f\|_\infty\lambda^{-k}\|\vf\|_{L^1}.
\]
Taking the sup over $\{\vf\in\cC^\infty\,:\,\|\vf\|_{L^1}=1\}$, it follows that $\nabla g\in L^\infty$, which implies $g\in W^{1,2}$. Hence, by Morrey's inequality, $g\in\cC^0$.
} 

In order to prove analyticity of the variance $\Var_\wp$ with respect to $\wp$, first notice that there is a positive lower bound on the spectral gap $\delta_\wp$ appearing in the proof of Corollary~\ref{cor:trivial} in a complex neighborhood of $[0,1]$. Thus, the series in \eqref{eq:var-series} converges uniformly in $\wp$. The partial sums are polynomials of $\wp$, hence the limit $\Var_\wp$ is an analytic function of $\wp$. 
\end{proof}

We finish the section with two simple but important results.
\begin{lem}\label{lem:eigenvalue}
Denoting $f_2(x,y):=f(x)-f(y)$, the operator $\cL^{(2)}_{i\lambda N^{-\frac 12}f_2,\wp}$ satisfies
\[
\left |\Leb_2 \bigl( \bigl [\cL^{(2)}_{i\lambda N^{-\frac 12}f_2,\wp} \bigr ]^N 1 \bigr)- e^{-\lambda^2\Var_\wp} \right| \leq C\frac{1+|\lambda|^3}{\sqrt{N}}.
\]
\end{lem}
\begin{proof}
The argument follows verbatim the previous discussion. Thus to prove the Lemma we only need to compute the second derivative of the leading eigenvalue, which we still denote $\mu_\nu$, and to show that $\mu_0''=2\Sigma_\wp^2$. In analogy with \eqref{eq:twoder},
\[
\mu''_0=\Leb_2( f_2^2)+2\sum_{n=1}^\infty\Leb_2\left( f_2  [\cL^{(2)}_{\wp} ]^n f_2\right),
\]
where $\Leb_2$ is the normalized Lebesgue measure on $\bT^4$. Then $ [\cL^{(2)}_{\wp} ]^n f_2(x,y)=\cL_\wp ^n f(x)-\cL_\wp^n f(y)$, $\Leb(1)=1$, and $\Leb(f)=0$ yield $\mu_0''=2\Sigma_\wp^2$. 
\end{proof}

\begin{lem}\label{lem:averaged-LDE}
There exists $L_0>0$ such that, for all $L\in (0,L_0)$, the following estimate holds,
\begin{equation}\label{eq:largedev}
\bfP_\wp\left(\left\{\left| \frac 1N\sum_{k=0}^{N-1}
f_{\omega,k}\right|\geq L\right\}\right)\leq Ce^{-CL^{2}N}.
\end{equation}
\end{lem}
\begin{proof}
This is an averaged large deviation estimate and can be obtained exactly
as the averaged CLT was obtained. Although the idea is standard we give here a sketch of the proof.
For any random variable $Y$, for each $\beta>0$,
\[
{\bf P}_\wp(\{Y\geq L\})\leq \bE_{{\bf P}_\wp}(\Id_{\{Y\geq L\}}e^{\beta(Y-L)})\leq e^{-\beta L}\bE_{{\bf P}_\wp}(e^{\beta Y}).
\]
Moreover, ${\bf P}_\wp(\{|Y|\geq L\})={\bf P}_\wp(\{Y\geq L\})+{\bf P}_\wp(\{Y\geq -L\})$. Applying such an inequality to the present situation we have 
\[
{\bf P}_\wp(\{Y\geq L\})\leq e^{-\beta L}\Leb(\cL_{\beta N^{-1} f,\wp}^N1).
\]
We again apply perturbation theory techniques at the beginning of this section to estimate the right-hand side. Using \eqref{eq:charfn} with $\nu=\beta N^{-1}, g=f, \phi_0 = 1$, we have
\[
\Leb(\cL_{\beta N^{-1} f,\wp}^N 1) =  \Leb\left(\mu_\nu^N\cQ_\nu 1+\cR_\nu^N1\right) = e^{ N\ln(\mu_\nu)}(1+\cO(|\beta| N^{-1}))+\cO(\rho^N).
\]
If we define the Legendre transform $I_C(L)=\sup_{|\nu|\leq C} L\nu-\ln\mu_\nu$ and we call $\nu_*$ the value in which the sup is attained, then choosing $\beta=\nu_* N$  we have
\[
{\bf P}_\wp(\{Y\geq L\})\leq (1+C|\nu_*|)e^{-N \cdot I_C(L)}+ \cO(\rho^N).
\]
To compute explicitly $I_C(L)$ we expand 
\[
L\nu- \ln \mu_\nu = L\nu-\frac 12  \nu^2 \mu_0'' + \cO(\nu^3).
\]
Minimizing this quadratic expression leads to a value of $\nu_* = \frac{L}{\mu_0''}$ and gives (recalling $\mu_0''=\Sigma_\wp^2$) the estimate,
\[
{\bf P}_\wp(\{Y\geq L\})\leq 2e^{-\frac{L^2 N}{2 \Sigma_\wp^2}(1 - \epsilon)}  .
\]
provided $L \leq C \epsilon$ where $C \epsilon$ is small. 
\end{proof}

\section{Quenched CLT}\label{sec:quenched}
Now that we have the CLT in average we would like to establish it for a
large class of sequences. Let $\Var_\wp$ be the variance of the average
CLT
with respect to the Bernoulli process with parameter $\wp$. We wish to
show
that for $\bP_\wp$ almost all sequences $\omega$ we have the CLT with
variance $\Var_\wp$.

To this end we start with an $L^2$ estimate: assuming that $Y_N$ is a
sequence of random variables such that $\bar Y:=\lim_N \bE_\wp(Y_N)$
exists
and is real, we can compute
\[
\bE_\wp(|Y_N-\bar Y|^2)= \bE_\wp(|Y_N|^2)-\bar Y^2+2\bar Y\,\Re (\bar
Y-\bE_\wp(Y_N)).
\]
Thus, recalling the notation $f_{\omega, k}:=f\circ T_{\omega_k}\circ \cdots \circ
T_{\omega_1}$ and the bound \eqref{eq:clt},
\begin{equation}\label{eq:clt1}
\begin{split}
\bE_\wp\left(\left|\Leb(e^{i\frac \lambda{\sqrt N}\sum_{k=0}^{N-1}
f_{\omega,k}})-e^{-\frac
12\lambda^2\Var_\wp}\right|^2\right)&=\bE_\wp\left(|\Leb(e^{i\frac
\lambda{\sqrt N}\sum_{k=0}^{N-1} f_{\omega,k}})|^2\right)\\
&\quad- e^{-\lambda^2\Var_\wp}+\cO\left(\frac{1+|\lambda|^3}{\sqrt N}\right).
\end{split}
\end{equation}
The first term on the right-hand side can be conveniently reinterpreted by
introducing a product system. That is, consider the maps
$T_{\omega_k}\oplus T_{\omega_k}:\bT^4\to\bT^4$,
which are represented by the block matrices
$\Bigl(\begin{smallmatrix}A_{\omega_k} & 0 \\ 0 &
A_{\omega_k}\end{smallmatrix}\Bigr)\in SL(4,\bN)$. 
Clearly they are hyperbolic toral automorphisms (although of a higher dimensional torus)
which leave Lebesgue measure invariant. The stable and unstable
directions are two dimensional. In perfect analogy with the averaged
CLT one can define $f_2(x,y):=f(x)-f(y)$ and study the operator
$
\cL^{(2)}_{i\lambda N^{-\frac 12}f_2,\wp}
$
(see Section~\ref{sec:gap}). A direct computation then shows that, calling $\Leb_2$ the normalized Lebesgue
measure on $\bT^4$, 
\[
\begin{split}
\bE_\wp\left(|\Leb(e^{i\frac \lambda{\sqrt N}\sum_{k=0}^{N-1}
f_{\omega,k}})|^2\right) & = \bE_\wp\left(\Leb_2\left (e^{i\frac
\lambda{\sqrt N}\sum_{k=0}^{N-1}
f_2\circ (T_{\omega_k}\oplus T_{\omega_k})\circ \cdots \circ
(T_{\omega_1}\oplus T_{\omega_1})} \right)\right)\\
& = \Leb_2 \bigl( \bigl [\cL^{(2)}_{i\lambda N^{-\frac 12}f_2,\wp} \bigr ]^N 1 \bigr).
\end{split}
\]
By Lemma~\ref{lem:eigenvalue} and by \eqref{eq:clt1},
\[
\bE_\wp\left(\left|\Leb(e^{i\frac \lambda{\sqrt N}\sum_{k=0}^{N-1}
f_{\omega,k}})-e^{-\frac 12
\lambda^2\Var_\wp}\right|^2\right) \leq C\frac{1+|\lambda|^3}{\sqrt N}.
\] 

By Chebyshev inequality the above estimate implies
\begin{equation}\label{eq:cheb}
\bP_\wp\left(\left\{\left|\Leb(e^{i\frac \lambda{\sqrt N}\sum_{k=0}^{N-1}
f_{\omega,k}})-e^{-\frac 12 \lambda^2\Var_\wp}\right|\geq
\ve\right\}\right)\leq C\ve^{-2}\frac{1+|\lambda|^3}{\sqrt N}.
\end{equation}
One would then like to prove almost sure convergence by applying a
Borel-Cantelli argument but two problems are in the way: on the one hand
the sum over $N$ of the above bound diverges, on the other hand one wants
the limit to hold almost surely for all $\lambda$, that is one has
potentially uncountably many sets to deal with. Both problems can be dealt
with by applying Borel-Cantelli to subsequences and then showing that
controlling the limit of such sequences one controls the limit for each
$N$ and $\lambda$.
First of all, notice that 
\begin{equation}\label{eq:mean}
\left|e^{i\frac \lambda{\sqrt N}\sum_{k=0}^{N-1} f_{\omega,k}}-e^{i\frac
{\lambda_1}{\sqrt N}\sum_{k=0}^{N-1} f_{\omega,k}}\right|
\leq \frac {|\lambda-\lambda_1|}{\sqrt N}\left|\sum_{k=0}^{N-1}
f_{\omega,k}\right|.
\end{equation}

On the other hand, notice that the estimate \eqref{eq:largedev} in
Lemma~\ref{lem:averaged-LDE} also implies
\begin{equation}\label{eq:N}
\begin{split}
&\bfP_\wp\left(\left\{\left|\frac 1{\sqrt N}\sum_{k=0}^{N-1}
f_{\omega,k}-\frac 1{\sqrt{N+M}}\sum_{k=0}^{N+M-1} f_{\omega,k}\right|\geq
\ve\right\}\right)\\
&=\bfP_\wp\left(\left\{\left|\frac {\sqrt{1+MN^{-1}}-1}{\sqrt
{N+M}}\sum_{k=0}^{N+M-1} f_{\omega,k}-\frac
{1}{\sqrt{N}}\sum_{k=N}^{N+M-1} f_{\omega,k}\right|\geq
\ve\right\}\right)\\
& \leq \bfP_\wp\left(\left\{\left|\frac 1{N+M}\sum_{k=0}^{N+M-1}
f_{\omega,k}\right|\geq
\frac{\ve}{2\sqrt{N+M}[\sqrt{1+MN^{-1}}-1]}\right\}\right)\\
&\quad+
\bfP_\wp\left(\left\{\left|\frac 1{M}\sum_{k=0}^{M-1}
f_{\omega,k}\right|\geq \frac{\ve \sqrt{N}}{2M}\right\}\right)\leq
Ce^{-CNM^{-1}\ve^2}.
\end{split}
\end{equation}
Next, consider $b\in(\frac 12,1)$ and the sets
\footnote{Here $[x]$ stands for the integer closest to $x$.}
$A_k:=\{2^k+[j 2^{bk}]\}_{j\leq 2^{(1-b)k}}$,
$\Lambda_k:=\{-k+jk^{-1}\}_{j\leq k^2}$ and $B_k:=A_k\times \Lambda_k$.
For each $(N,\lambda)\in B_k$ let
$\Delta_k(N,\lambda)=\{(N_1,\lambda_1)\in\bN\times\bR\;:\; |N-N_1|\leq
2^{bk}+1,\, |\lambda-\lambda_1|\leq k^{-1}\}$.
Clearly
\[
\bigcup_{(N,\lambda)\in B_k}\Delta_k(N,\lambda) \supset
\left\{(N,\lambda)\in\bN\times\bR\;:\; 2^k\leq N\leq 2^{k+1},\,
|\lambda|\leq k\right\}=:J_k.
\]
We can then write
\[
\begin{split}
&\bP_\wp\left(\left\{\sup_{(N,\lambda)\in J_k}\left|\Leb(e^{i\frac
\lambda{\sqrt N}\sum_{l=0}^{N-1}f_{\omega,l}})-e^{-\frac 12
\lambda^2\Var_\wp}\right|
\geq 4\ve\right\}\right)\\
&\leq \sum_{(N,\lambda)\in
B_k}\bP_\wp\left(\left\{\sup_{(N_1,\lambda_1)\in
\Delta_k(N,\lambda)}\left|\Leb(e^{i\frac {\lambda_1}{\sqrt
N_1}\sum_{l=0}^{N_1-1}f_{\omega,l}})-e^{-\frac 12
\lambda_1^2\Var_\wp}\right|
\geq 4\ve\right\}\right)\\
&\leq\sum_{(N,\lambda)\in B_k} \Biggl\{
\bP_\wp\left(\left\{\left|\Leb(e^{i\frac {\lambda}{\sqrt
N}\sum_{l=0}^{N-1}f_{\omega,l}})-e^{-\frac 12
\lambda^2\Var_\wp}\right|\geq
\ve\right\}\right) \\
&\quad+\bP_\wp\left(\left\{\Leb\left(\sup_{(N_1,\lambda_1)\in
\Delta_k(N,\lambda)}\left|e^{i\frac {\lambda_1}{\sqrt
N_1}\sum_{l=0}^{N_1-1}f_{\omega,l}}-e^{i\frac {\lambda_1}{\sqrt
N}\sum_{l=0}^{N-1}f_{\omega,l}}\right|\right)\geq \ve\right\}\right)\\
&\quad+\bP_\wp\left(\left\{\Leb\left(\sup_{(N_1,\lambda_1)\in
\Delta_k(N,\lambda)}\left|e^{i\frac {\lambda_1}{\sqrt
N}\sum_{l=0}^{N-1}f_{\omega,l}}-e^{i\frac {\lambda}{\sqrt
N}\sum_{l=0}^{N-1}f_{\omega,l}}\right|\right)\geq \ve\right\}\right)
\Biggr\},
\end{split}
\]
where we have assumed $\Sigma_\wp k^{-1}\leq \ve$ in order to deal with
the
difference $e^{-\frac 12 \lambda^2\Var_\wp}-e^{-\frac 12
\lambda_1^2\Var_\wp}$.
 For each bounded function $g\geq 0$ holds
\[
{\bf P}_\wp(\{g \geq A\})=\bE_{{\bf P}_\wp}(\Id_{\{g \geq A\}})\geq \bE_{\wp}(\Leb(\Id_{\{g \geq A\}})\Id_{\{\Leb(g) \geq 2A\}}).
\]
But $\Leb(g)\leq |g|_\infty\Leb(\{g\geq A\})+A$, and $\Leb (g)\geq 2A$
implies $\Leb(\{g\geq A\})\geq A|g|_\infty^{-1}$. Thus,
\[
\bP_\wp(\{\Leb(g)\geq 2A\}) \leq A^{-1} |g|_\infty{\bf P}_\wp(\{g\geq
A\}).
\]
We can estimate the above expression by
\[
\begin{split}
&\bP_\wp\left(\left\{\sup_{(N,\lambda)\in J_k}\left|\Leb(e^{i\frac
\lambda{\sqrt N}\sum_{l=0}^{N-1}f_{\omega,l}})-e^{-\frac 12
\lambda^2\Var_\wp}\right|
\geq 4\ve\right\}\right)\\
&\leq \sum_{(N,\lambda)\in
B_k}\Biggl[\bP_\wp\left(\left\{\left|\Leb(e^{i\frac {\lambda}{\sqrt
N}\sum_{l=0}^{N-1}f_{\omega,l}})-e^{-\frac 12 \lambda^2\Var_\wp}\right|
\geq \ve\right\}\right) \\
&\quad+4\ve^{-1}{\bf P}_\wp\left(\left\{\sup_{(N_1,\lambda_1)\in
\Delta_k(N,\lambda)}\left|e^{i\frac {\lambda_1}{\sqrt
N_1}\sum_{l=0}^{N_1-1}f_{\omega,l}}-e^{i\frac {\lambda_1}{\sqrt
N}\sum_{l=0}^{N-1}f_{\omega,l}}\right|\geq \frac \ve 2\right\}\right)\\
&\quad+4\ve^{-1}{\bf P}_\wp\left(\left\{\sup_{(N_1,\lambda_1)\in
\Delta_k(N,\lambda)}\left|e^{i\frac {\lambda_1}{\sqrt
N}\sum_{l=0}^{N-1}f_{\omega,l}}-e^{i\frac {\lambda}{\sqrt
N}\sum_{l=0}^{N-1}f_{\omega,l}}\right|\geq \frac \ve 2\right\}\right)
\Biggr],
\end{split}
\]
Thus, remembering \eqref{eq:mean},
\[
\begin{split}
&\bP_\wp\left(\left\{\sup_{(N,\lambda)\in J_k}\left|\Leb(e^{i\frac
\lambda{\sqrt N}\sum_{l=0}^{N-1}f_{\omega,l}})-e^{-\frac 12
\lambda^2\Var_\wp}\right|
\geq 4\ve\right\}\right)\\
&\leq \sum_{(N,\lambda)\in
B_k}\Biggl[\bP_\wp\left(\left\{\left|\Leb(e^{i\frac {\lambda}{\sqrt
N}\sum_{l=0}^{N-1}f_{\omega,l}})-e^{-\frac 12 \lambda^2\Var_\wp}\right|
\geq \ve\right\}\right) \\
&\quad+4\ve^{-1}{\bf P}_\wp\left(\left\{\sup_{(N_1,\lambda_1)\in
\Delta_k(N,\lambda)}|\lambda_1|\,\left|\frac 1{\sqrt
N_1}\sum_{l=0}^{N_1-1}f_{\omega,l}-\frac {1}{\sqrt
N}\sum_{l=0}^{N-1}f_{\omega,l}\right|\geq \frac \ve 2\right\}\right)\\
&\quad+4\ve^{-1}{\bf P}_\wp\left(\left\{\sup_{(N_1,\lambda_1)\in
\Delta_k(N,\lambda)}\frac {|\lambda_1-\lambda|}{\sqrt
N}\left|\sum_{l=0}^{N-1}f_{\omega,l}\right|\geq \frac \ve 2\right\}\right)
\Biggr]\\
&\leq \sum_{(N,\lambda)\in
B_k}\Biggl[\bP_\wp\left(\left\{\left|\Leb(e^{i\frac {\lambda}{\sqrt
N}\sum_{l=0}^{N-1}f_{\omega,l}})-e^{-\frac 12 \lambda^2\Var_\wp}\right|
\geq \ve\right\}\right) \\
&\quad+4\ve^{-1}\sum_{|N_1-N|\leq 2^{bk}+1}{\bf
P}_\wp\left(\left\{\left|\frac
1{\sqrt N_1}\sum_{l=0}^{N_1-1}f_{\omega,l}-\frac {1}{\sqrt
N}\sum_{l=0}^{N-1}f_{\omega,l}\right|\geq \frac \ve {4k}\right\}\right)\\
&\quad+4\ve^{-1}{\bf P}_\wp\left(\left\{\frac1{\sqrt
N}\left|\sum_{l=0}^{N-1}f_{\omega,l}\right|\geq \frac {k\ve}
2\right\}\right) \Biggr].
\end{split}
\]
Then the estimates \eqref{eq:largedev}, \eqref{eq:N} and \eqref{eq:cheb}
imply, for $k\geq \Sigma_\wp\ve^{-1}$, 
\[
\begin{split}
&\bP_\wp\left(\left\{\sup_{(N,\lambda)\in J_k}\left|\Leb(e^{i\frac
\lambda{\sqrt N}\sum_{l=0}^{N-1}f_{\omega,l}})-e^{-\frac 12
\lambda^2\Var_\wp}\right|
\geq 4\ve\right\}\right)\\
&\leq C\sum_{(N,\lambda)\in
B_k}\Biggl[\ve^{-2}\frac{1+|\lambda|^3}{\sqrt N}+\ve^{-1}\sum_{|N_1-N|\leq
2^{bk}+1}e^{-C \frac{N}{|N-N_1|}\ve^2k^{-2}}+ \ve^{-1}
e^{-Ck^2\ve^2}\Biggr]\\
&\leq Ck^2 2^{(1-b)k}\ve^{-1}\left[\ve^{-1} k^32^{-\frac
k2}+2^{bk}e^{-C2^{k(1-b)}\ve^2k^{-2}}+e^{-C\ve^2 k^2}\right],
\end{split}
\]
for which it follows that the sum over $k$ is finite. By Borel-Cantelli it
follows that the above events $\left\{\sup_{(N,\lambda)\in
J_k}\left|\Leb(e^{i\frac \lambda{\sqrt
N}\sum_{l=0}^{N-1}f_{\omega,l}})-e^{-\frac 12 \lambda^2\Var_\wp}\right|
\geq 4\ve\right\}$ happen only finitely many times with probability one.
That is, for each $\ve>0$, there exists a random variable
$N_{\ve}:\Omega\to \bN\cup \{\infty\}$, $\bP_\wp$-almost surely finite,
such
that
\[
\sup_{|\lambda|\leq \log_2N}\left|\Leb(e^{i\frac \lambda{\sqrt
{N}}\sum_{k=0}^{N-1} f_{\omega,k}})-e^{-\frac 12 \lambda^2\Var_\wp}\right|
\leq \ve \quad\text{for}\quad N\geq N_{\ve}.
\]
Here we used the fact that, for each fixed $N$, $|\lambda|\leq \log_2N$
implies $(N,\lambda)\in J_{\lfloor{\log_2 N}\rfloor}$.
Let us call $\widetilde \Omega_\ve$ the bad set of sequences, involving
$N_{\ve}=\infty$. It is an increasing set with decreasing $\ve$, such that
$\bP_\wp(\bigcup_{\ve>0} \widetilde \Omega_\ve)=\lim_{\ve\downarrow 0}
\bP_\wp(\widetilde \Omega_\ve)=0$; the bad set is independent of $\ve$.

This concludes the proof and establishes the almost sure CLT where almost
sure means that, fixing any Bernoulli measure, the set of the sequences for
which we do not have CLT has
zero measure. Note, however, that
the limit (more precisely, the variance) is not constant but depends on
$\wp$. This is natural since the deterministic limits $\wp=0$ and $\wp=1$ generically have different variances and as $\wp$ varies, the variance should interpolate smoothly between these two extremal values, which indeed is confirmed by Lemma~\ref{lem:variance}.

\end{document}